\documentclass[12pt, a4paper, oneside]{amsart}
\usepackage{amssymb, amsmath, amsthm, verbatim, amsbsy,cite}
\usepackage[cp1251]{inputenc}
\usepackage[english]{babel}
\usepackage{amsaddr}
\newtheorem{theorem}{Theorem}
\newtheorem{lemma}{Lemma}
\numberwithin{lemma}{section}
\newtheorem{prop}{Proposition}
\numberwithin{prop}{section}
\newtheorem{conj}{Conjecture}
\numberwithin{equation}{section}

\newcommand{\Aut}{\operatorname{Aut}}
\newcommand{\Out}{\operatorname{Out}}
\newcommand{\Outdiag}{\operatorname{Outdiag}}

\begin{document}

\title[On groups isospectral to simple groups]{On the structure of finite\\ groups isospectral to finite simple groups}

\author{Mariya A. Grechkoseeva}
\address{Sobolev Institute of Mathematics and Novosibirsk State University,\\ Ac. Koptyuga 4, Novosibirsk, 630090, Russia.}
\email{gma@math.nsc.ru}

\author{Andrey V. Vasil$'$ev}
\address{Sobolev Institute of Mathematics and Novosibirsk State University,\\ Ac. Koptyuga 4, Novosibirsk, 630090, Russia.}
\email{vasand@math.nsc.ru}

\thanks{The work is supported by the Russian Foundation for Basic Research (project 13-01-00505).}

\begin{abstract}
Finite groups are said to be isospectral if they have the same sets of element orders. A finite nonabelian simple group $L$ is said to be almost recognizable by
spectrum if every finite group isospectral to $L$ is an almost simple group with socle isomorphic to $L$. It is known that all finite simple sporadic,
alternating and exceptional groups of Lie type, except $J_2$, $A_6$, $A_{10}$ and $^3D_4(2)$, are almost recognizable by spectrum. The present paper is the final step in the proof
of the following conjecture due to V.D. Mazurov: there exists a positive integer $d_0$ such that every finite simple classical group of dimension larger than $d_0$ is almost recognizable by spectrum.
Namely, we prove that a nonabelian composition factor of a~finite group isospectral to a finite simple symplectic or orthogonal group $L$ of dimension at least 10,
is either isomorphic to $L$ or not a group of Lie type in the same characteristic as $L$, and combining this result with earlier work, we deduce that Mazurov's conjecture holds with $d_0=60$.


\end{abstract}

\maketitle

\section{Introduction}

The {\em spectrum} $\omega(G)$ of a finite group $G$ is the set of element orders of $G$. Groups are {\it isospectral} if they have the same spectra. Given a finite group $G$ with nontrivial soluble radical, one can construct infinitely many different finite groups isospectral to $G$ \cite{94Shi,98Maz.t}. In contrast, there is a conjecture due to V.D. Mazurov that in general the set of groups isospectral to a finite nonabelian simple group $L$ is finite and consists of groups closely related to $L$. More precisely, Mazurov conjectured\footnote{in his talk on International Algebraic Conference in St. Petersburg, September 24–29, 2007.} that if $L$ is an alternating group of sufficiently large degree or a simple group of Lie type of sufficiently large Lie rank, and $G$ is a group isospectral to $L$, then $G$ is an almost simple group with socle isomorphic to $L$. For brevity we refer to a nonabelian simple group $L$ such that every finite group isospectral to $L$ is an almost simple group with socle isomorphic to $L$, as {\em almost recognizable by spectrum}.

Gorshkov \cite{13Gor.t} proved that all alternating groups of degree at least 11 are almost recognizable by spectrum, and recently Vasil'ev and Staroletov \cite{14VasSt.t} completed the investigation of almost recognizability of exceptional groups. Thus Mazurov's conjecture has been reduced to the following: there exists a positive integer $d_0$ such that every simple classical group of dimension larger than $d_0$ is almost recognizable. The main purpose of the present paper is to prove this conjecture with $d_0=60$. As a result, we establish the following theorem (our notation for nonabelian simple groups follows \cite{85Atlas}).

\begin{theorem}\label{t:main} Let $L$ be one of the following nonabelian simple groups:
\begin{enumerate}
\item \label{t:sporadic} a sporadic group other than $J_2$;
\item \label{t:alternating} an alternating group $A_n$, where $n\neq6,10$;
\item \label{t:exceptional} an exceptional group of Lie type other than ${}^3D_4(2)$;
\item\label{t:linear} $L_n(q)$, where $n\geqslant45$ or $q$ is even;
\item\label{t:unitary} $U_n(q)$, where $n\geqslant45$, or $q$ is even and $(n,q)\neq (4,2),(5,2)$;
\item\label{t:BnCn} $S_{2n}(q),O_{2n+1}(q)$, where either $q$ is odd and $n\geqslant28$, or $q$ is even and $n\geqslant20$;
\item\label{t:Dn} $O^+_{2n}(q)$, where either $q$ is odd and $n\geqslant31$, or $q$ is even and $n\geqslant20$;
\item\label{t:2Dn} $O^-_{2n}(q)$, where either $q$ is odd and $n\geqslant30$, or $q$ is even and $n\geqslant20$.
\end{enumerate}
Then every finite group isospectral to $L$ is isomorphic to some group $G$ with $L\leqslant G\leqslant \Aut L$. In particular,
there are only finitely many pairwise nonisomorphic finite groups isospectral to $L$.
\end{theorem}

Theorem~\ref{t:main} is undoubtedly a sum of efforts by numerous mathematicians, and a com\-prehen\-sive list of references covering its proof is too long to be given here.
But for every series of simple  groups mentioned in Theorem~\ref{t:main}, we cite the work in which the proof for this series was completed, and as a rule this work includes a survey of previous investigations.
Thus, see already mentioned \cite{13Gor.t} and \cite{14VasSt.t} for alternating and exceptional groups. See \cite{98MazShi} for sporadic groups, \cite[Theorem~1]{15Vas} for linear and unitary groups in odd characteristic, and \cite{08VasGr.t} and  \cite[Corollary 2]{11Gr1.t} for linear and unitary groups in characteristic 2. The present paper is concerned with symplectic and orthogonal groups.

By \cite[Theorem~2 and Proposition 6]{15Vas}, if $L$ is a finite simple symplectic or orthogonal group as in (\ref{t:BnCn})--(\ref{t:2Dn}) of Theorem~\ref{t:main} and $G$ is a finite group with $\omega(G)=\omega(L)$, then
$G$ has only one nonabelian composition factor and this factor $S$ is a symplectic or orthogonal group having the same underlying characteristic as $L$. If $S\simeq L$, then
the soluble radical of $G$ is trivial \cite[Theorem~1.1]{15Gr}, and hence $G$ is an almost simple group with socle isomorphic to $L$, as required. If $S\not\simeq L$, then by \cite[Theorem 3]{09VasGrMaz.t}
there are at most two possibilities for $S$ (see Lemma~\ref{l:previous} below). We eliminate these possibilities by the following theorem, and thus Theorem~\ref{t:main} follows.

\begin{theorem} \label{t:quasi}
Let $q$ be a power of a prime $p$, $L$ one of the groups $S_{2n}(q)$, where $n\geqslant 2$ and
$(n,q)\not\in\{(2,2),(2,3)\}$, $O_{2n+1}(q)$, where $n\geqslant3$, or $O^\pm_{2n}(q)$, where $n\geqslant 4$, and let $G$ be a finite group with $\omega(G)=\omega(L)$.  Suppose that some nonabelian composition factor $S$ of $G$ is a group of Lie type
over a field of characteristic $p$. Then either $S\simeq L$ or  one of the following holds:
\begin{enumerate}
\item \label{i:s4} $L=S_4(q)$, where $q\neq3^{2k+1}$, and $S=L_2(q^2)$;

\item \label{i:s8} $L\in\{O_9(q), S_8(q)\}$ and $S=O^-_8(q)$;

\item \label{i:s6} $\{L,G\}=\{O^+_8(2), S_6(2)\}$;

\item \label{i:o7} $\{L,G\}=\{O^+_8(3), O_7(3)\}$.
\end{enumerate}

\end{theorem}

We conclude the introduction with a discussion of nonabelian simple groups that are genuinely not almost recognizable by spectrum. The groups $J_2$, $A_6$, $A_{10}$, $^3D_4(2)$, $U_4(2)\simeq S_4(3)$, and $U_5(2)$ are not
almost recognizable by spectrum; moreover, each of them is isospectral to a group with nontrivial soluble radical (see \cite{98MazShi,91BrShi, 98Maz.t, 13Maz.t}). All the groups $L_2(q)$,  $L_3(q)$  and $U_3(q)$, except for $L_2(9)$, $L_3(3)$, $U_3(5)$ and $U_3(p)$ where $p$ is a Mersenne prime and $p^2-p+1$ is a prime too, are almost recognizable (see \cite{94BrShi}, \cite{00MazXuCao.t,04Zav} and \cite{00MazXuCao.t,06Zav.t} respectively). The question is still open for $L_4(q)$ and $U_4(q)$ with $q$ odd: the group $L_4(13^{24})$ was proved to be not almost recognizable, and it was conjectured that there are infinitely many $q$ such that $L_4(q)$ is not almost recognizable \cite{08Zav2}. Thus, the exceptions in (\ref{t:sporadic})--(\ref{t:exceptional}) of Theorem \ref{t:main} are necessary, while the condition $n\geqslant 45$ in (\ref{t:linear}) and (\ref{t:unitary}) can probably be replaced by $n\geqslant 5$.

The simple group $S_4(q)$ is almost recognizable if and only if $q=3^m$ with $m>1$ odd; moreover, the case (\ref{i:s4}) of Theorem \ref{t:quasi} is possible \cite{00MazXuCao.t,02Maz.t}.
It is known that $\omega(O^+_8(2))=\omega(S_6(2))$ and $\omega(O^+_8(3))=\omega(O_7(3))$ and there are no other finite groups with such spectra \cite{97Maz.t, 97ShiTan}. The group $S_8(2)$ is not almost recognizable \cite{06MazMog}.
In the last section  we generalize this result to all groups $S_8(q)$ with $q$ even, showing the case (\ref{i:s8}) of Theorem \ref{t:quasi} is possible when $q$ is even. We suspect that it is also possible when $q$ is odd (and this is proved for $L=O_9(q)$ in \cite{14GrSt}). Thus it seems likely that all the exceptions appearing in Theorem \ref{t:quasi} can be realised. On the other hand, we believe that all other symplectic and orthogonal groups are almost recognizable by spectrum.

To summarize, we propose the following

\begin{conj}\label{conj:main} Let $L$ be one of the following groups:
\begin{enumerate}
\item $L_n(q)$, where $n\geqslant5$;
\item $U_n(q)$, where $n\geqslant5$ and $(n,q)\neq(5,2)$;
\item $S_{2n}(q)$, where $n\geqslant3$, $n\neq4$ and $(n,q)\neq(3,2)$;
\item $O_{2n+1}(q)$, where $q$ is odd, $n\geqslant3$, $n\neq4$ and $(n,q)\neq(3,3)$;
\item $O_{2n}^\varepsilon(q)$, where $n\geqslant4$ and $(n,q,\varepsilon)\neq (4,2,+),(4,3,+)$.
\end{enumerate}
Then every finite group isospectral to $L$ is isomorphic to some group $G$ with $L\leqslant G\leqslant \Aut L$.
\end{conj}

In fact, to prove Conjecture~\ref{conj:main} it suffices to show that the only nonabelian composition factor of a group isospectral to a simple classical group $L$ under consideration is not a group of Lie type whose underlying characteristic differs from that of $L$. The existing generic proof of this assertion recently obtained in \cite{15Vas} requires that the dimension of $L$ is as large as stated in Theorem~\ref{t:main}. It is worth also mentioning
that the assertion was proved in some special cases, in particular for many classical groups with disconnected prime graph (see \cite{09VasGorGr.t, 09AleKon.t, 09Kon, 09HeShi,10SheShiZin.t,12HeShi, 12GrLyt.t, 12ForIraAha, 13ForIraAha} for recent research in this area).

\section{Preliminaries}

By $(a_1,a_2,\dots,a_k)$ and $[a_1,a_2,\dots,a_k]$ we denote respectively the greatest common divisor and least common multiple of
positive integers $a_1,a_2,\dots,a_k$. If $a$ is a positive integer and $r$ is a prime then $\pi(a)$ denotes the set of prime divisors of $a$ and $(a)_r$ denotes the highest power of $r$ that divides $a$.

\begin{lemma}[Zsigmondy \cite{Zs}]\label{l:zsig}
Let $q\geqslant 2$ and $n\geqslant 3$ be integers with $(q,n)\neq (2,6)$. There exists
a prime $r$ such that $r$ divides $q^n-1$ but does not divide $q^i-1$ for $i<n$.
\end{lemma}

With notation of Lemma \ref{l:zsig}, we call a prime $r$ a {\em primitive prime divisor} of $q^n-1$ and denote it by $r_n(q)$.
By $R_n(q)$ we denote the set of all primitive prime divisors of $q^n-1$. Observe that $R_n(q)\subseteq R_n(q^k)$ if $k$ is coprime to $n$, and $R_{nk}(q)\subseteq R_n(q^k)$ for all $n$ and $k$.

Given a group $G$, we set $\pi(G)=\pi(|G|)$ and define the {\em prime graph} $GK(G)$ as follows: its vertex set is $\pi(G)$ and two different primes $r$ and $s$ are adjacent if and only if $G$ has an element of order $rs$.
We use standard graph-theoretic terminology: a {\em coclique} of a graph is a set of pairwise nonadjacent vertices; a {\em neighbourhood} of a vertex $v$ of a graph is the subgraph consisting of all vertices adjacent to $v$ and all edges connecting two such vertices.

\begin{lemma}\label{l:structure}
Let $L$ be a finite nonabelian simple group of Lie type other than $L_3(3)$, $U_3(3)$, $S_4(3)\simeq U_4(2)$ and let $G$ be a finite group with $\omega(G)=\omega(L)$.
Then the following hold.
\begin{enumerate}
\item There is a nonabelian simple group $S$ such that $S\leqslant \overline G=G/K\leqslant \Aut S,$ where $K$ is
the soluble radical of $G$.

\item If $\rho$ is a coclique of size at least $3$ in $GK(G)$, then at most one prime of $\rho$ divides
$|K|\cdot|\overline{G}/S|$.

\item If $r\in\pi(G)$ is not adjacent to $2$ in $GK(G)$, then $r$ is coprime to $|K|\cdot|\overline{G}/S|$.

\end{enumerate}
\end{lemma}

\begin{proof} If there is a coclique of size 3 in $GK(L)$, then the assertion is the main theorem of \cite{05Vas.t} supplemented with \cite{09VasGor.t}  and \cite[Theorem 7.1]{05VasVd.t}.
If there are no cocliques of size 3 in $GK(L)$, then we have that $GK(L)$ is disconnected by \cite{11VasVd.t, 81Wil, 89Kon.t}.  Then
the Gruenberg--Kegel theorem \cite[Theorem A]{81Wil} implies that either (i) and (iii) holds true for $G$, or $G$ is a Frobenius or 2-Frobenius group ($G$ is called 2-Frobenius if $G=ABC$, where $A$, $AB$ are normal in $G$, $B$ is normal in $BC$, and $AB$ and $BC$ are Frobenius groups).
Simple groups of Lie type that can be isospectral to a Frobenius or 2-Frobenius group are described in \cite{03Ale.t}, and these groups are precisely  $L_3(3)$, $U_3(3)$ and $S_4(3)\simeq U_4(2)$.
\end{proof}

We say that a finite group $H$ is a (proper) {\it cover} of a finite group $G$ if there is a (nontrivial)  normal subgroup $K$ of $H$ such that $H/K\simeq G$.

\begin{lemma}[{\cite[Lemma 2.3]{15Gr}}] \label{l:reduction}
Let $A$ and $B$ be finite groups. The following are equivalent.
\begin{enumerate}
\item $\omega(H)\not\subseteq\omega(B)$ for any proper cover $H$ of $A$;
\item $\omega(H)\not\subseteq\omega(B)$ for any split extension $H=K:A$, where $K$ is a nontrivial elementary abelian group.
\end{enumerate}
\end{lemma}

\begin{lemma}[{\cite[Lemma 1]{97Maz.t}}] \label{l:action} Let $G$ be a finite group,
$K$ a normal subgroup of $G$ and $G/K~$ a Frobenius group with kernel $N$ and cyclic complement~$C.$ If
$(|N|, |K|)=1$ and $N$ is not contained in $KC_G(K)/K$, then $r|C|\in\omega(G)$ for some $r\in\pi(K)$.
\end{lemma}

\begin{lemma}[{\cite[Lemma 2.7]{15Gr}}]\label{l:hh} Let $S$ be a finite simple group of Lie type over a field of characteristic $p$ and let $S$ act faithfully on a
vector space  $V$ over a field of characteristic $r$, where $r\neq p$. Let $H=V\rtimes S$ be a natural semidirect product of $V$ by $S$.
Suppose that $s$ is a power of $r$ and some proper parabolic subgroup $P$ of $S$ contains an element of order~$s$. If the unipotent radical of $P$ is abelian, or  both $p$ and $r$ are odd, or $p=2$ and $r$ is not a Fermat prime, or $r=2$ and $p$ is not a Mersenne prime, then $rs\in\omega(H)$.
\end{lemma}

We conclude with several lemmas on spectra of symplectic and orthogonal groups. If $p\in\pi(G)$, then $\omega_{p'}(G)$ denotes the subset of $\omega(G)$ consisting of numbers coprime to $p$.  In Lemmas \ref{l:spectrum_s_odd} and \ref{l:spectrum_s_2}, $\pm$ in $[a_1\pm 1,\dots, a_s\pm 1]$ means that we can choose $+$ or $-$ for every entry independently. In Lemma \ref{l:semi}, for brevity, we write $\varepsilon$ instead of $\varepsilon1$ for $\varepsilon\in\{+,-\}$.

\begin{lemma}[{\cite[Corollaries 2 and 6]{10But.t}}]\label{l:spectrum_s_odd}
Let $L$ be one of the simple groups $S_{2n}(q)$ or  $O_{2n+1}(q)$, where $n\geqslant 2$ and $q$ is a power of an odd prime $p$. Let $d=1$ if $L=S_{2n}(q)$ or $n=2$, and let  $d=2$ if $L=O_{2n+1}(q)$ with $n\geqslant 3$.
Then $\omega(L)$ consists of all divisors of the following numbers:
\begin{enumerate}
 \item $(q^n\pm1)/2$;
 \item $[q^{n_1}\pm 1,\dots,q^{n_s}\pm 1]$, where $s\geqslant 2$, $n_i>0$ for all $1\leqslant i\leqslant s$ and $n_1+\dots+n_s=n$;
 \item $p^k(q^{n_1}\pm 1)/d$, where $k,n_1>0$ and $p^{k-1}+1+2n_1=2n$;
 \item $p^k[q^{n_1}\pm 1,\dots, q^{n_s}\pm1]$, where $k>0$, $s\geqslant 2$,  $n_i>0$ for all $1\leqslant i\leqslant s$ and $p^{k-1}+1+2(n_1+\dots+n_s)=2n$;
 \item $p^k$ if $2n=p^{k-1}+1$ for some $k > 0$.
\end{enumerate}
\end{lemma}

\begin{lemma}[{\cite[Corollary 3]{10But.t}}]\label{l:spectrum_s_2}
Let $L=S_{2n}(q)$, where $n\geqslant 2$ and $q$ is even. Then $\omega(L)$ consists of all divisors of the following numbers:
\begin{enumerate}
 \item $[q^{n_1}\pm 1,\dots,q^{n_s}\pm 1]$, where $s\geqslant 1$, $n_i>0$ for all $1\leqslant i\leqslant s$ and $n_1+\dots+n_s=n$;
 \item $2[q^{n_1}\pm 1,\dots,q^{n_s}\pm 1]$, where $s\geqslant 1$, $n_i>0$ for all $1\leqslant i\leqslant s$ and $n_1+\dots+n_s=n-1$;
 \item $2^k[q^{n_1}\pm 1,\dots, q^{n_s}\pm1]$, where $k\geqslant 2$, $s\geqslant 1$,  $n_i>0$ for all $1\leqslant i\leqslant s$ and $2^{k-2}+1+n_1+\dots+n_s=n$;
 \item $2^k$ if $n=2^{k-2}+1$ for some $k\geqslant 2$.
\end{enumerate}
\end{lemma}

\begin{lemma}[{\cite[Proposition 3.1(5)]{05VasVd.t}}]\label{l:adj_p}
Let $L=O_{2n}^\pm(q)$, where $n\geqslant 4$ and $q$ is a power of a prime $p$. If $r\in \pi(L)\cap R_k(q)$, where $k$ is odd and $k>n-2$, or $k$ is even and $k/2>n-2$,  then $rp\not\in\omega(L)$.
\end{lemma}

\begin{lemma}\label{l:semi}
Let $L=O_{2n}^\varepsilon(q)$, where $n\geqslant 4$, $\varepsilon\in\{+,-\}$ and $q$ is a power of a prime $p$. Then $\omega_{p'}(L)$ consists of all divisors of the following numbers:
\begin{enumerate}
 \item $(q^n-\varepsilon)/(4,q^n-\varepsilon)$;
 \item $[q^{n_1}-\delta, q^{n_2}-\varepsilon\delta]/d$, where $\delta\in\{+,-\}$, $n_1,n_2>0$, $n_1+n_2=n$;  $d=2$ if $(4,q^n-\varepsilon)=4$, $(q^{n_1}-\delta)_2=(q^{n_2}-\varepsilon\delta)_2$,
 and $d=1$ otherwise;
 \item $[q^{n_1}-\delta_1, q^{n_2}-\delta_2, \dots, q^{n_s}-\delta_s]$, where  $s\geqslant 3$, $\delta_i\in\{+,-\}$, $n_i>0$ for all $1\leqslant i\leqslant s$, $n_1+\dots+n_s=n$ and $\delta_1\delta_2\dots\delta_s=\varepsilon$.
\end{enumerate}

\end{lemma}
\begin{proof}
The assertion follows, for example, from \cite[Theorem 6]{07ButGr.t}.
\end{proof}

\begin{lemma} \label{l:spectrum_o8p}
Let $L=O_8^+(q)$, where $q$ is a power of a prime $p$. Then  $\omega(L)$ consists of all divisors of the following numbers:
\begin{enumerate}
 \item $(q^4-1)/(2,q-1)^2$, $(q^3\pm1)/(2,q-1)$, $q^2-1$, $p(q^2\pm1)/(2,q-1)$;
 \item $p^2(q\pm 1)/(2,q-1)$ if $p=2,3$;
 \item $25$ if $p=5$;
 \item $8$ if $p=2$.
\end{enumerate}

\end{lemma}

\begin{proof}
The assertion follows from \cite[Corollaries 4 and 9]{10But.t}.
\end{proof}

\begin{lemma}  \label{l:spectrum_o8m}
Let $L=O_8^-(q)$, where $q$ is even. Then  $\omega(L)$ consists of all divisors of the following numbers:
$q^4\pm1$, $(q^2\pm q+1)(q^2-1)$, $2(q^2+1)(q\pm 1)$, $4(q^2-1)$, and $8$.
\end{lemma}

\begin{proof}
The assertion follows from \cite[Corollary 4]{10But.t}.
\end{proof}

\begin{lemma}\label{l:diff} Let $q$ be a power of a prime $p$.
\begin{enumerate}
 \item \label{i:bncn_diff} Let $q$ be odd and $n\geqslant 3$. Let $r=r_{2n-2}(q)$ if $q^{n-1}\equiv 1\pmod 4$ and $r=r_{n-1}(q)$ if $q^{n-1}\equiv -1\pmod 4$. Then
 $2pr\in\omega(S_{2n}(q))\setminus \omega(O_{2n+1}(q))$.
 \item \label{i:dnbn} If $n\geqslant 4$ and $(n,q)\neq (4,2)$, then $pr_{2n-2}(q)\in\omega(O_{2n+1}(q))\setminus\omega({}O^-_{2n}(q))$.
 If $n\geqslant 4$ is even, then $pr_{n-1}(q)\in\omega(O_{2n+1}(q))\setminus\omega({}O^-_{2n}(q))$.
 \item \label{i:s6o8} If $q>3$, then $(q^4-1)/(2,q-1)^2\in\omega(O^+_8(q))\setminus\omega(S_6(q))$.
 \item \label{i:o8s6} If $q$ is odd, then $p(q^2+1)\in\omega(S_{6}(q))\setminus\omega(O^+_{8}(q))$.
 \item \label{i:bncn}$\omega(O_{2n+1}(q))\subseteq\omega(S_{2n}(q))$ for all $n\geqslant 2$.
 \item \label{i:orth} $\omega(O_{2n-1}(q))\subseteq \omega(O_{2n}^\pm(q))\subseteq \omega(O_{2n+1}(q))$ for all $n\geqslant 3$.

\end{enumerate}
\end{lemma}
\begin{proof} (\ref{i:bncn_diff}) Let $q^{n-1}\equiv \varepsilon \pmod 4$. Then $2pr$ divides $p(q^{n-1}+\varepsilon)$, and in particular it belongs to $\omega(S_{2n}(q))$ by Lemma \ref{l:spectrum_s_odd}.
Suppose that $2pr\in \omega(O_{2n+1}(q))$. It follows by Lemma \ref{l:spectrum_s_odd} that $2pr$ divides either $p^k(q^{n_1}\pm 1)/2$ for some $k\geqslant 1$ and $n_1>0$ with $p^{k-1}+1+2n_1=2n$, or $p^k[q^{n_1}\pm 1,\dots,q^{n_s}\pm 1]$ for some $k\geqslant 1$,  $s\geqslant 2$, $n_1,\dots,n_s>0$ with $p^{k-1}+1+2n_1+\dots+2n_s=2n$.  In fact, by the definition of primitive divisor,  it cannot divide a number of the latter form since all $n_1,\dots,n_s$ are less than $n-1$.
And if $2pr$ divides $p^k(q^{n_1}+\tau)/2$  with $p^{k-1}+1+2n_1=2n$ and $\tau=\pm1$, then by the definition of primitive divisor we have that $k=1$, $n_1=n-1$ and $\tau=\varepsilon$. But then $(q^{n_1}+\tau)/2$ is odd, a contradiction.

(\ref{i:dnbn}) See \cite[Proposition 3.1]{05VasVd.t}.

(\ref{i:s6o8}) Let $a=(q^4-1)/(2,q-1)^2$. Lemma \ref{l:spectrum_o8p} implies that $a\in\omega(O^+_8(q))$. By Lemmas \ref{l:spectrum_s_odd} and \ref{l:spectrum_s_2}, the orders of semisimple elements of $S_6(q)$ are precisely divisors of $(q^3\pm1)/(2,q-1)$, $(q^2+1)(q\pm1)/(2,q-1)$ and $q^2-1$.
 It is clear that $a$ divides none of $(q^3\pm1)/(2,q-1)$ and $q^2-1$.
Furthermore, $$a=(q^2+1)\cdot \frac{(q+1)}{(2,q-1)}\cdot\frac{(q-1)}{(2,q-1)}.$$ Since $q>2$, both $(q-1)/(2,q-1)$ and $(q+1)/(2,q-1)$ are greater than one, and so $a$ does not divide $(q^2+1)(q\pm1)/(2,q-1)$ either.

(\ref{i:o8s6})--(\ref{i:bncn}) See Lemmas \ref{l:spectrum_s_odd} and \ref{l:spectrum_o8p}.

(\ref{i:orth}) It is well known that $O_{2n-1}(q)<O_{2n}^\pm(q)<O_{2n+1}(q)$, and the assertion follows.
\end{proof}

\begin{lemma} \label{l:s6} Let $q$ be even, $S=S_6(q)$ and $L=O_8^+(q)$. Suppose that $V$ is a nontrivial $S$-module over a field of characteristic $2$ and $H$ is a natural semidirect product of $V$ and $S$. Then $\omega(H)\not\subseteq \omega(L)$.
\end{lemma}
\begin{proof}
In the proof of Lemma 4.1 in \cite{15Gr}, it was established that  $\omega(H)$ contains at least
one of the numbers $16$, $24$, and $2(q^2+q+1)$. By Lemma \ref{l:spectrum_o8p}, none of these numbers belong to $\omega(L)$.
\end{proof}

\section{Proof of Theorem \ref{t:quasi}}

As we mentioned in the Introduction, the starting point for our proof of Theorem \ref{t:quasi} is the~following assertion.

\begin{lemma}\label{l:previous}
Let $q$ be a power of a prime $p$, $L$ one of the groups $S_{2n}(q)$, where $n\geqslant 2$ and
$(n,q)\not\in\{(2,2),(2,3)\}$, $O_{2n+1}(q)$, where $n\geqslant3$, or $O^\pm_{2n}(q)$, where $n\geqslant 4$, and let $G$ be a finite group with $\omega(G)=\omega(L)$.
Suppose that some nonabelian composition factor $S$ of $G$ is a group of Lie type
over a field of characteristic $p$. If $S\not\simeq L$, then one of the following holds:
\begin{enumerate}
\item $L=S_4(q)$ and $S=L_2(q^2)$;

\item $\{L, S\}\subseteq \{S_6(q), O_7(q),  O^+_8(q)\}$;

\item $\{L, S\}\subseteq \{S_{2n}(q), O_{2n+1}(q), {}O^-_{2n}(q)\}$ and $n\geqslant 4$;

\item $L=O^+_{2n}(q)$, $S\in\{S_{2n-2}(q), O_{2n-1}(q)\}$, and $n\geqslant 6$ is even.
\end{enumerate}
\end{lemma}

\begin{proof}
The assertion is a combination of \cite[Theorem 3]{09VasGrMaz.t} and \cite[Theorem 1]{12Sta.t}.
\end{proof}

Let $q=p^m$ and let $L$ be one of the groups $S_{2n}(q)$, where $n\geqslant 2$ and
$(n,q)\not\in\{(2,2),(2,3)\}$, $O_{2n+1}(q)$, where $n\geqslant3$, or $O^\pm_{2n}(q)$, where $n\geqslant 4$. Let $G$ be a finite group with $\omega(G)=\omega(L)$. By Lemma \ref{l:structure}, we have
$$S\leqslant G/K\leqslant \Aut S,$$ where $K$ is the soluble radical of $G$ and $S$ is a nonabelian simple group. Suppose that
$S$ is a group of Lie type over a field of characteristic $p$. By Lemma \ref{l:previous}, either $S\simeq L$, or $L$ and $S$ are as in the conclusion of Lemma \ref{l:previous}.

Let $L=S_4(q)$, where $q>3$. If $q=3^{m}$ and $m$ is odd, then $G\simeq L$ by \cite{02Maz.t}, and so $S\simeq L$. If $q\neq 3^{m}$ with $m$ odd, then $S=L_2(q^2)$ by Lemma \ref{l:previous}, and this is (\ref{i:s4}) of the conclusion of  Theorem \ref{t:quasi}. To examine the other cases, we need some auxiliary results.

\begin{lemma}\label{l:out}
Let $L\neq S_4(q)$. If $k>2$ and $r\in R_{mk}(p)$, then  $r\not\in\pi(\Out S)$.
\end{lemma}

\begin{proof}
Since $S$ is a symplectic or orthogonal group over a field of order $q=p^m$, it follows that
$\pi(\Out S)\subseteq \{2,3\}\cup\pi(m)$. By Fermat's Little Theorem $mk$ divides $r-1$ and so $r>3$ and $r>m$.
\end{proof}

\begin{lemma}\label{l:bncn}
If $L=S_{2n}(q)$, where $n\geqslant 3$ and $q$ is odd, then $S\neq O_{2n+1}(q)$.
\end{lemma}

\begin{proof}
Assume the contrary. It follows by \cite[Proposition 1.3]{15Gr} that $K=1$ in this case, and hence $S\leqslant G\leqslant \Aut S$. Let  $r=r_{2(n-1)m}(p)$ if $q^{n-1}\equiv 1 \pmod 4$ and $r=r_{(n-1)m}(p)$ otherwise. By Lemma \ref{l:diff}, there is an element $g$ of order $2pr$ in $G\setminus S$. By Lemma \ref{l:out}, we have $r\not\in\pi(G/S)$.

Suppose that $p\in\pi(G/S)$. Then $G$ contains a field automorphism of $S$ of order $p$. By \cite[Proposition 4.9.1(a)]{98GorLySol}, the centralizer of this automorphism in $S$ includes $O_{2n+1}(q_0)$, where $q=q_0^p$.
Therefore $pr_{2n}(q_0)$ and $p[q_0+1, q_0^{n-1}\pm 1]$ lie in $\omega(G)$. If $p$ does not divide $n$, then $r_{2n}(q_0)\in R_{2n}(q)$ and so $pr_{2n}(q_0)\not\in\omega(L)$ by \cite[Proposition 3.1]{05VasVd.t}.
Let $p$ divide $n$. Then $r_{2n-2}(q_0)\in R_{2n-2}(q)$. By the definition of primitive divisor and Lemma \ref{l:spectrum_s_odd},
it follows that $p(q_0+1)r_{2n-2}(q_0)\in\omega(L)$ if and only if $q_0+1$ divides $q^{n-1}+1$, which implies that $n$ is even. But then $q_0+1$ does not divide $q^{n-1}-1$, and by the same reasoning we conclude that
$p(q_0+1)r_{n-1}(q_0)\not\in\omega(L)$. In any case, $\omega(G)\not\subseteq\omega(L)$, a contradiction.

Thus $p,r\not\in\pi(G/S)$, and it follows that $g^{pr}\not\in S$ and $g^2\in S$.  Therefore $G\setminus S$ contains an involution $t=g^{pr}$ such that $pr\in\omega(C_S(t))$.
Suppose that $t\in \operatorname {Inndiag}S$. Then $\operatorname {Inndiag}S\leqslant G$ and since $\operatorname {Inndiag}S\simeq SO_{2n+1}(q)$, it follows that $q^n+1\in\omega(G)\setminus\omega(L)$, which is impossible.
Therefore, by \cite[Proposition 4.9.1(d)]{98GorLySol}, we have that $t$ is a field automorphism,  and hence $m$ is even. In particular $q\equiv 1\pmod 4$ and $r=r_{2(n-1)m}(p)$.
By  \cite[Prop 4.9.1(a,b)]{98GorLySol}, the centralizer $C_S(t)$ can be embedded into $SO_{2n+1}(q_0)$, where $q=q_0^2$, and so $C_S(t)$ has no elements of order $r$, a contradiction.
\end{proof}

\begin{lemma}\label{l:action2}  Suppose $\pi(G)$ contains four different primes $p$, $s$, $r$, and $w$ with the following properties:

\begin{enumerate}
 \item $S$ contains a Frobenius subgroup $F$ whose kernel is a $p$-group and whose complement has order $s$;
 \item $r\in \pi(K)$;
 \item $\{s,r,w\}$ is a coclique in $GK(G)$.
\end{enumerate}
Then $p\in\pi(K)$ and $pt\in \omega(G)$ for every $t\in\pi(S)\setminus\{p\}$.
\end{lemma}

\begin{proof} Since $\{s,r,w\}$ is a coclique in $GK(G)$ and $r\in \pi(K)$, it follows by Lemma \ref{l:structure}(ii) that $s\not\in\pi(K)$.
Construct a normal $r$-series of $K$ as follows:
$$1=R_0\leqslant K_1\leqslant R_1\leqslant ...\leqslant K_{l-1}\leqslant R_{l-1}\leqslant K_l\leqslant R_l=K,$$ where  $K_i/R_{i-1}=O_{r'}(K/R_{i-1})$ and $R_i/K_i=O_r(K/K_i)$ for $1\leqslant i\leqslant l$.

Suppose first that $K/K_l\neq 1$ and let $\tilde K=G/K_l$ and $\widetilde G=G/K_l$. Since the group $C_{\widetilde G}(\tilde K)\tilde K/\tilde K$ is a normal subgroup of $\tilde G/\tilde K \simeq G/K$,
it either contains $S$ or is trivial. In the former case $C_{\widetilde G}(\tilde K)$ has an element of order $s$, and so $rs\in\omega(G)$, contrary to (iii). In the latter case we apply Lemma \ref{l:action} to the Frobenius group $F$, and again obtain $rs\in\omega(G)$.

Now suppose that $K=K_l$ and let $\tilde R=R_{l-1}/K_{l-1}$, $\tilde K=K/K_{l-1}$, and $\tilde G=G/K_{l-1}$. Since $O_{r'}(\tilde K)=1$, it follows by \cite[Lemma 1.2.3]{56HalHig} that $C_{\tilde K}(\tilde R)\leqslant \tilde R$.
Furthermore, we may assume that $C_{\widetilde G}(\tilde R)\leqslant \tilde K$ as above, and therefore $C_{\widetilde G}(\tilde R)\leqslant \tilde R$. If $p$ does not divide $|\tilde K|$, then $(|\tilde K|,|F|)=1$ and the Schur--Zassenhaus theorem
implies that $\tilde G$ has a subgroup isomorphic to $F$. In this case we apply Lemma \ref{l:action} and deduce that $rs\in\omega(G)$, a contradiction. Thus $p$ divides $|\tilde K|$. Let  $\tilde P$ be a Sylow $p$-subgroup of $\tilde K$,  $\tilde Z=Z(\tilde P)$ and $\tilde N=N_{\tilde G}(\tilde P)$. By the Frattini argument, $\tilde G/\tilde K=\tilde N\tilde K/\tilde K$, and therefore $\tilde N$ has an element $g$ of order $s$. Furthermore, since $C_{\tilde N}(\tilde Z)$ is normal in $\tilde N$, this group either is contained in $\tilde K$ or has $S$ as a section. In the former case $g$ does not centralize $\tilde Z$, and then $\tilde Z=[\tilde Z,\langle g\rangle]\times C_{\tilde Z}(g)$ yields $[\tilde Z,\langle g\rangle]\neq 1$. Thus $[\tilde Z,\langle g\rangle]\rtimes\langle g\rangle$ is a Frobenius group that acts on $\tilde R$ faithfully, and applying Lemma \ref{l:action} once again, we have $rs\in\omega(G)$, a contradiction. Thus  $C_{\tilde N}(\tilde Z)$ has $S$ as a section, and so $pt\in\omega(G)$ for every $t\in\pi(S)\setminus\{p\}$, as required.
\end{proof}

\begin{lemma} \label{l:adj_s}
Let $L=S_{2n}(q)$ or $L=O_{2n+1}(q)$, where $n\geqslant 5$. Suppose that $r\in\pi(L)\cap R_k(q)$ with $k>2$, $b\in\omega_{p'}(L)$, and $r$ divides $b$.
\begin{enumerate}
 \item If $k=2n$ then $b$ divides $(q^n+1)/(2,q-1)$.
 \item If $k=2n-2$ then $b$ divides $[q^{n-1}+1,q\pm1]$.
 \item If $n$ is even and $k=n-1$ then $b$ divides $[q^{n-1}-1,q\pm1]$.
\end{enumerate}
\end{lemma}

\begin{proof}
The assertion follows from the definition of primitive prime divisor and  Lemmas \ref{l:spectrum_s_odd} and \ref{l:spectrum_s_2}.
\end{proof}

\begin{lemma} \label{l:adj_o}

Let $L=O_{2n}^+(q)$, where $n\geqslant 6$ is even. Suppose that $r\in\pi(L)\cap R_k(q)$ with $k>2$, $b\in\omega_{p'}(L)$, and $r$ divides $b$.
\begin{enumerate}
 \item If $k=2n-2$ then $b$ divides $q^{n-1}+1$.
 \item If $k=n-1$ then $b$ divides $q^{n-1}-1$.
 \end{enumerate}
\end{lemma}

\begin{proof}
By the definition of primitive prime divisor and  Lemma \ref{l:semi}, it follows that $b$ divides $[q^{n-1}+1,q+1]=q^{n-1}+1$ in (i) and $[q^{n-1}-1,q-1]=q^{n-1}-1$ in (ii).
\end{proof}

Now we return to the proof of Theorem \ref{t:quasi}.

\begin{lemma}\label{l:o8} If $L,S\in\{S_6(q), O_7(q),  O^+_8(q)\}$, then either $S\simeq L$, or $\{L,G\}=\{O^+_8(2), S_6(2)\}$, or $\{L,G\}=\{O^+_8(3), O_7(3)\}$.
\end{lemma}
\begin{proof}
Assume that $S\not\simeq L$. If $q=2$ or $q=3$, then  respectively $\{L,G\}=\{S_6(2),O_8^+(2)\}$ or $\{L,G\}=\{O_7(3),O_8^+(3)\}$ by \cite{97ShiTan, 02Maz.t}.

Let $q>3$. Since $\omega(S)\subseteq \omega(L)$, it follows from Lemma \ref{l:diff}\,(iii,v) that $S\neq O_8^+(q)$.
Furthermore, if $q$ is odd then $S\neq S_6(q)$ by Lemma \ref{l:diff}\,(i,iv). Also by Lemma \ref{l:bncn}, if $L=S_6(q)$, then $S\neq O_7(q)$. Thus
$L=O^+_8(q)$ and $S=O_7(q)$ (if $q$ is even we write $O_7(q)$ instead of $S_6(q)$). In particular, $\pi(S)=\pi(L)$.

By Lemmas \ref{l:spectrum_s_odd}, \ref{l:spectrum_s_2}, and \ref{l:spectrum_o8p},
the orders of semisimple elements of $L$ and $S$ are precisely the divisors of \begin{equation}\label{e:sso8}(q^4-1)/(2,q-1)^2, (q^3\pm1)/(2,q-1), q^2-1\end{equation} and
\begin{equation}\label{e:sso7}(q^2+1)(q\pm1)/(2,q-1), (q^3\pm1)/(2,q-1), q^2-1 \end{equation} respectively. In particular, if $rr_6(q)\in\omega(L)$  for some $r\in\pi(L)$, then
$r$ divides $(q^3+1)/(2,q-1)$,  and if $rr_3(q)\in\omega(L)$  for some $r\in\pi(L)$, then
$r$ divides $(q^3-1)/(2,q-1)$. Since $(q^3+1)/(2,q-1)$ and $(q^3-1)/(2,q-1)$ are coprime, it follows that $r_6(q)$ and $r_3(q)$ are not adjacent and have disjoint
neighbourhoods in $GK(L)$. Furthermore, $pr_6(q),pr_3(q)\not\in\omega(L)$ by Lemma \ref{l:adj_p}.

We claim that $K\neq 1$ yields $\omega(G)\not\subseteq\omega(L)$. By Lemma \ref{l:reduction} we may assume that
$K$ is an elementary abelian $r$-group for some prime $r$ and $G$ has a subgroup that is isomorphic to  a semidirect product of $K$ and $S$. Also we may assume that
$S$ acts on $K$ faithfully. Otherwise $S$ centralizes $K$, and hence all primes of $\pi(S)=\pi(L)=\pi(G)$ other than $r$ are adjacent to $r$ in $GK(G)$, contrary to the fact
that $r_3(q)$ and $r_6(q)$ have disjoint neighbourhoods in $GK(L)$.

Let $r\neq p$. Since $r_3(q)$ divides the order of a proper parabolic subgroup of $S$ with Levi factor of type $A_2$ and $pr_3(q)\not\in\omega(S)$, the group $S$ has
a Frobenius subgroup whose kernel is a $p$-group and whose complement has order $r_3(q)$. Furthermore, $S$ has
a Frobenius subgroup with kernel of order $q^2$ and cyclic complement of order $(q^2-1)/(2,q-1)$ by \cite[Lemma 5]{11Gr}. Applying Lemma~\ref{l:action}, we conclude that $G$ has elements of orders $rr_3(q)$ and $r(q^2-1)/(2,q-1)$.
If $r(q^2-1)/(2,q-1)\in\omega(L)$, then keeping in mind that $q>3$ and consulting (\ref{e:sso8}), we see that either $r$ divides $(q^2+1)/(2,q-1)$ or $r=2$. In the former case $rr_3(q)\not\in\omega(L)$. Let $r=2$.
Then the highest power of $2$ in $\omega(L)$ is equal to $(q^2-1)_2$. Let us consider a parabolic subgroup $P$ of $SO_7(q)$ with Levi factor of type $B_2$. The Levi factor of $P$ is
$A\times B$ where $A\simeq GL_1(q)$, $B\simeq SO_5(q)$ and both $A$ and $B$ contain elements with spinor norm a non-square (see \cite[p. 98]{90KlLie}). Since $SO_5(q)$ has elements of order $q^2-1$,
it follows that $P\cap S$ also has elements of order $q^2-1$. Furthermore, the Chevalley commutator formula implies that the unipotent radical of $P$ is abelian (see, for example, \cite[Lemma 2.2]{92RRS}).
Applying Lemma \ref{l:hh} we obtain that $2(q^2-1)_2\in\omega(G)\setminus\omega(L)$.

Now let $r=p$. If $p=2$ then we apply Lemma \ref{l:s6}. If $p$ is odd then $pr_3(q)\in\omega(G)$ by \cite[Lemma 3.2]{15Gr}, and so  $\omega(G)\not\subseteq\omega(L)$.

Thus $K=1$ and hence $S\leqslant G\leqslant \Aut S$. Consulting (\ref{e:sso8}) and (\ref{e:sso7}), we see that $r_{4m}(p)(q^2-1)/(2,q-1)$ lies in $\omega(G)\setminus\omega(S)$. Since $r_{4m}(p)\not\in\pi(G/S)$ by Lemma \ref{l:out}, it follows from (\ref{e:sso7}) that at least one
of the numbers $(q+1)/(2,q-1)$ and $(q-1)/(2,q-1)$ belongs to $\omega(G/S)$. The group $\Out S$ is a direct product of $\Outdiag S$ of order $(2,q-1)$ and a cyclic group of order $m$.
If $q$ is odd then $q^3+1\in\omega(\operatorname{Inndiag} S)\setminus\omega(L)$, and so $G\cap \operatorname{Inndiag} S=S$ for both even and odd $q$. Therefore the exponent of $G/S$ divides $m$.
However $(q\pm1)/(2,q-1)=(p^m\pm1)/(2,p-1)>m$ for $q>3$, a contradiction.
\end{proof}

\begin{lemma}
If $L,S\in \{S_{2n}(q), O_{2n+1}(q), {}O^-_{2n}(q)\}$, where $n\geqslant 4$, then either $S\simeq L$, or $L\in\{O_{9}(q),S_{8}(q)\}$ and $S={}O^-_{8}(q)$.
\end{lemma}
\begin{proof}
Assume that $S\not\simeq L$. Since $\omega(S)\subseteq\omega(L)$, it follows by Lemma \ref{l:diff}\,(i,vi) that $S\neq S_{2n}(q)$ and by Lemma \ref{l:diff}\,(ii) that $(L,S)\neq (O_{2n+1}(q), O_{2n}^-(q))$.
Also $(L,S)\neq (S_{2n}(q), O_{2n+1}(q))$ by Lemma \ref{l:bncn}. Therefore $L\in\{O_{2n+1}(q),S_{2n}(q)\}$ and $S={}O^-_{2n}(q)$.
If $n=4$, the proof is complete.

Suppose that $n\geqslant 5$. Let $s=r_{2n-2}(q)$. Then $s$ divides the order of a parabolic subgroup of $S$ with Levi factor of type ${}^2D_{n-1}$ and $ps\not\in\omega(S)$ by Lemma \ref{l:adj_p}.
So $S$ has a Frobenius subgroup $F$ whose kernel is a $p$-group and whose complement has order $s$. For every $n$ we will find $r,w\in \pi(L)$ such that $r\in\pi(K)$ and $\{s,r,w\}$ is a coclique
in $GK(L)$, and then will apply Lemma \ref{l:action2} to the group $F$ and the primes $r$ and $w$.

Let $n$ be odd and let $r=r_{nm}(p)$. Since $r\in\pi(L)\setminus\pi(S)$ and
$r\not\in\pi(\Out S)$ by Lemma \ref{l:out}, it follows that $r\in\pi(K)$. Exploiting Lemma \ref{l:adj_s}, we deduce that $\{s,r,r_{2n}(q)\}$ is a coclique in $GK(L)$: both
$r$ and $s$ do not divide $q^n+1$, and $r$ does not divide $q^{n-1}+1$ nor $q\pm 1$.

Let $n$ be even. Let $r_1=r_{(n-2)m}(p)$,  $r_2=r_{(n+2)m}(p)$ if $(n,q)\neq (8,2)$ and $r_1=r_3(2)$, $r_2=r_5(2)$ otherwise.
Then $r_1r_2\in\omega(L)\setminus\omega(S)$ by Lemmas \ref{l:spectrum_s_2}, \ref{l:spectrum_s_odd} and \ref{l:semi}. Furthermore, $r_1,r_2\not\in\pi(\Out S)$ by Lemma \ref{l:out}.
Therefore at least one of $r_1$ and $r_2$ divides $|K|$. Denote this number by $r$. Since $n\geqslant 6$, both $n+2$ and $n-2$ do not divide $2n-2$, and so $r$ is not adjacent to $s$ nor $r_{n-1}(q)$ in $GK(L)$ by
Lemma \ref{l:adj_s}. By the same lemma, $s$ and $r_{n-1}(q)$ are not adjacent either, and $\{s,r, r_{n-1}(q)\}$ is the desired coclique.

Applying Lemma \ref{l:action2}, we conclude that $pt\in \omega(G)$ for all $t\in\pi(S)$, $t\neq p$. But $r_{2n}(q)\in\pi(S)$ and $pr_{2n}(q)\not\in\omega(L)$ by Lemmas \ref{l:spectrum_s_odd} and \ref{l:spectrum_s_2}, a contradiction.
\end{proof}

\begin{lemma}
If $L= O^+_{2n}(q)$, where $n\geqslant 6$ is even, then $S\not\in\{S_{2n-2}(q), O_{2n-1}(q)\}$.
\end{lemma}

\begin{proof}
Assume the contrary. Let $s=r_{n-1}(q)$ and $w=r_{2n-2}(q)$. Then
$s$ divides the order of a parabolic subgroup of $S$ with Levi factor of type $A_{n-2}$ and $ps\not\in\omega(S)$, therefore,
$S$ contains a Frobenius group whose kernel is a $p$-group and whose complement has order $s$. Furthermore, Lemma \ref{l:adj_o} implies that $s$ and $w$ are not adjacent in $GK(L)$.
Let $r_1=r_{(n-2)m}(p)$,  $r_2=r_{(n+2)m}(p)$ if $(n,q)\neq (8,2)$, and $r_1=r_3(2)$, $r_2=r_5(2)$ otherwise. Then $r_1r_2\in\omega(L)\setminus\omega(S)$ by Lemmas \ref{l:spectrum_s_2}, \ref{l:spectrum_s_odd} and \ref{l:semi}, and
$r_1,r_2\not\in\pi(\Out S)$ by Lemma \ref{l:out}. Thus at least one of the numbers $r_1$ and $r_2$ divides $|K|$. Denote this number by $r$. Both $n+2$ and $n-2$ do not divide $2n-2$, and so $\{s, r, w\}$ is a coclique in  $GK(L)$ by Lemma \ref{l:adj_o}. By Lemma \ref{l:action2}, we have $pr_{2n-2}(q)\in \omega(G)$. On the other hand,  $pr_{2n-2}(q)\not\in \omega(L)$ by Lemma \ref{l:adj_p}, a contradiction.
\end{proof}

We have considered all the cases in the conclusion of Lemma \ref{l:previous}, and so the proof of Theorem \ref{t:quasi} is complete.

\section{New examples of non-quasirecognizable simple  groups}

A finite nonabelian simple group $L$ is said to be {\it quasirecognizable by spectrum} if every finite group isospectral to $L$ has only one nonabelian composition factor and this factor is isomorphic to $L$. Clearly
quasirecognizability is a necessary condition for being almost recognizable.

Recall that for even $q$, the simple group $S_{2n}(q)$ is equal to $Sp_{2n}(q)$ and the simple group $O^\varepsilon_{2n}(q)$ is $\Omega^\varepsilon_{2n}(q)$, a subgroup of index 2 in $GO_{2n}^\varepsilon(q)$
(see, for example, \cite[p. xii]{85Atlas}). Mazurov and Moghaddamfar \cite{06MazMog} noted that $\omega(Sp_8(2))=\omega(GO_8^-(2))$, and thereby $Sp_8(2)$ is not quasirecognizable by spectrum. We generalize this result to all even $q$.

\begin{prop}\label{l:GO8}
Let $q$ be even. Then $\omega(Sp_8(q))=\omega(GO_8^-(q))$ and in particular $Sp_8(q)$ is not quasirecognizable by spectrum.
\end{prop}
\begin{proof}
Denote $Sp_8(q)$, $GO_8^-(q)$, and $\Omega_8^-(q)$ by $L$, $G$, and $S$ respectively. Since $G$ is a subgroup in $GO_9(q)\simeq Sp_8(q)$, it follows that $\omega(G)\subseteq\omega (L)$. Thus it suffices to show that $\omega(L)\setminus\omega(S)\subseteq\omega(G)$. It follows from Lemma \ref{l:spectrum_s_2} that $\omega(L)$ consists of all divisors of the~numbers $q^4\pm 1$, $(q^2\pm q+1)(q^2-1)$,  $2(q^3\pm1)$,
$2(q^2+1)(q\pm 1)$, $4(q^2\pm 1)$, and $8(q\pm 1)$. Supplementing this by Lemma \ref{l:spectrum_o8m}, we see that every number in $\omega(L)\setminus\omega (S)$ is a divisor of one of the numbers $2(q^3\pm1)$, $4(q^2+1)$, and $8(q\pm 1)$.

It is well known that $GO_{2n}^-(q)$ has a subgroup of the form $GO_{2k}^\varepsilon(q)\times GO_{2n-2k}^{-\varepsilon}(q)$ for every $1\leqslant k\leqslant n-1$ and $\varepsilon\in\{+,-\}$. 
Thus $G$ contains  $GO_6^+(q)\times GO_2^-(q)$, $GO_4^+(q)\times GO_4^-(q)$, and  $GO_2^+(q)\times GO_6^-(q)$.

It can be verified by means of \cite{GAP} that $4\in\omega(GO_4^+(2))$. Since $GO_4^+(2)\leqslant GO_4^+(q)$ and $SL_2(q^2)<GO_4^-(q)$,  it follows that $4(q^2+1)\in\omega(G)$.

Since $q^3-\varepsilon \in\omega(GO_6^\varepsilon(q))$ and $GO_2^\varepsilon(q)$ is  dihedral of order $2(q-\varepsilon)$, we conclude that $2(q^3\pm 1)\in\omega(G)$.
Furthermore, $S_8\simeq GO_6^+(2)\leqslant GO_6^+(q)$, and hence $8(q+1)\in\omega(G)$. It remains to establish that
$GO_6^-(q)$ has an element of order 8.

The group $GO_6^-(q)$ is isomorphic to a split extension of $SU_4(q)$ by a graph automorphism of order 2. We identify $SU_4(q)$ with
$$H=\{A\in SL_4(q^2)\mid AJ\overline A^\top=J\},$$
where $\overline{(a_{ij})}=(a_{ij}^q)$ and $J$ is the $4\times 4$ matrix with 1's on the antidiagonal and 0's elsewhere. Then $\gamma$ defined by $A^\gamma=\overline A$
is a graph automorphism of $H$. Choose $t\in GF(q^2)$ such that $t^q\neq t$ and let
$$B=
\begin{pmatrix}
1& t & 0& 0\\
0& 1 & 1 & 0 \\
0& 0 & 1 & t^q \\
0& 0 & 0 & 1 \\
\end{pmatrix}.$$
It is easy to verify that $B\in H$ and
$$(B\gamma)^4=
\begin{pmatrix}
1& 0 & 0& t^2+t^{2q}\\
0& 1 & 0 & 0 \\
0& 0 & 1 & 0 \\
0& 0 & 0 & 1 \\
\end{pmatrix}.$$
Since $t^2+t^{2q}\neq 0$, the order of $B\gamma$ is equal to 8, and hence $8\in\omega(GO_6^-(q))$.
\end{proof}

It is worth noting that another example of a finite group isospectral to $S_8(q)$ with $q$ even and having $O_8^-(q)$ as a composition factor
was constructed in \cite{14GrSt}.

{\bf Acknowledgments.} We are grateful to the referee for his careful reading of the article and helpful comments and suggestions.

\end{document}